\documentclass{article}
\usepackage{amsmath, amsfonts, graphicx, amsthm,color,enumerate}

 \usepackage{tikz}

\usepackage{cite}

\numberwithin{equation}{section}

\DeclareMathOperator{\Inn}{Inn}
\DeclareMathOperator{\Aut}{Aut}
\DeclareMathOperator{\diag}{diag}
\DeclareMathOperator{\id}{id}
\DeclareMathOperator{\ch}{char}
\DeclareMathOperator{\GL }{GL }
\DeclareMathOperator{\SL}{SL}
\DeclareMathOperator{\Sp}{Sp}

\DeclareMathOperator{\G}{G}
\DeclareMathOperator{\C}{C}
\DeclareMathOperator{\F}{F}
\DeclareMathOperator{\E}{E}
\DeclareMathOperator{\Spin}{Spin}
\DeclareMathOperator{\Mat}{Mat}
\DeclareMathOperator{\Skew}{Skew}
\DeclareMathOperator{\Tr}{Tr}
\DeclareMathOperator{\tr}{tr}
\DeclareMathOperator{\sr}{sr}
\DeclareMathOperator{\Sr}{Sr}
\DeclareMathOperator{\A}{A}
\DeclareMathOperator{\e}{e}
\DeclareMathOperator{\B}{B}

\newcommand{\I}{\mathcal{I}}

\title{Isomorphism classes of $k$-involutions of algebraic groups of type $\F_4$}
\author{John Hutchens  \\ \emph{jdhutchens@saumag.edu} \\ \\ \emph{Department of Mathematics and Computer Science} \\ \emph{Southern Arkansas University} \\ \emph{Magnolia, AR  71753} }

\begin{document}
\maketitle

\begin{abstract}
We continue the classification of isomorphism classes of $k$-involutions of exceptional algebraic groups.  In this paper we classify $k$-involutions for split groups of type $\F_4$ over certain fields, and their fixed point groups.  The classification of $k$-involutions is equivalent to the classification of symmetric $k$-varieties, \cite{HWa93}.
\end{abstract}

\section{Introduction}

This is a continuation of a classification initiated by Helminck et al. \cite{He00,Hu12,HW02,DHW06,DHW-}.  In \cite{He00} Helminck introduces three invariants whose classification is equivalent to a full classification of $k$-involutions of reductive algebraic groups.  Here we look at the specific case when $G$ is a split algebraic group of type $\F_4$.  One reason to classify the isomorphism classes of $k$-involutions for an algebraic group of a certain type has to do with their correspondence to symmetric $k$-varieties of these groups.  A \emph{symmetric $k$-variety} is a quotient space of the form $G(k)/H(k)$ where $G$ is a reductive algebraic group, $H = G^{\theta}$ is the fixed point group of $\theta$ an automorphism of order $2$ and $G(k)$ (respectfully $H(k)$) are the $k$-rational points of $G$ (respectfully $H$).  For a torus $S \subset G$ we denote by $X^*(S)$ the group of characters, and $\Phi(S)$ the root space.  We define the set $I_k(S_{\theta}^-)$ in section $2$, along with other definitions needed to make the following statement.  Helminck shows that the classification of such spaces can be reduced to the classification of the following invariants,

\begin{enumerate}[(1)]
\item classification of admissible involutions of $(X^*(T),X^*(S), \Phi(T), \Phi(S))$, where $T$ is a maximal torus in $G$, S is a maximal $k$-split torus contained in $T$
\item classification of the $G(k)$-isomorphism classes of $k$-involutions of the $k$-anisotropic kernel of $G$
\item classification of the $G(k)$-isomorphism classes of $k$-inner elements $a\in I_k(S_{\theta}^-)$,
\end{enumerate}

Yokota gave explicit descriptions of $k$-involutions and their fixed point groups for algebraic groups of type $\F_4$ for $k=\mathbb{C}, \mathbb{R}$. Over $k= \mathbb{C}$ and $\mathbb{R}$ our results correspond to the $\gamma$ and $\sigma$ maps in \cite{Yo90}.  Here we use different methods, and fit the results into the theory described in \cite{He00}.

Aschbacher and Seitz give a full classification of $k$-involutions when $k$ is of even characteristic in \cite{AS76}. They provide isomorphism classes of $k$-involutions, and their centralizers, which we refer to as their fixed point groups.  Over fields of odd characteristic these results are also known, and can be found in \emph{The Classification of Finite Simple Groups} by Gorenstein, Lyons, and Solomon \cite{GLS}.  The study of such automorphisms and their relation to algebraic groups was initiated by Gantmacher in \cite{Ga39} in order to classify real simple Lie groups.  In \cite{Be55, Be57} Berger classified involutions of real groups and their symmetric spaces, which was also done by Helminck in \cite{He88}.

In the $\F_4$ case we end up with two main types of isomorphism classes of invariants of type (1).  According to \cite{He00} the representatives of one isomorphism class of $k$-involutions should send every element of a $k$-split maximal torus to its inverse, and the representatives of a second isomorphism class fixes a rank $3$ $k$-split torus.

We give the full classification of $k$-involutions when $k=K, \mathbb{R}, \mathbb{Q}_p$ and $\mathbb{F}_q$ when $p \geq 2$ and $q>2$, where $K$ is the algebraic closure of $k$.  We finish by giving descriptions of the fixed point groups of isomorphism classes of $k$-involutions, and discussing the interpretation of the isomorphism classes in terms of Galois cohomology and their relation to Kac coordinates.

I would like to thank E. Neher, S. Garibaldi, H. Petersson, and V. Chernousov, all of whom I met at the Field's Institute workshop for exceptional groups and algebras, for their advice and conversations about all things Jordan and exceptional.  I would particularly like to thank H. Petersson for the ongoing email correspondence we have had since that workshop.

\section{Preliminaries and recollections}

Most of our notation is borrowed from \cite{Sp98} for algebraic groups, \cite{He00} for $k$-involutions and symmetric $k$-varieties, \cite{Se97} for Galois cohomology, and \cite{SV00} and \cite{McC05} for Albert algebras and composition algebras. \

The letter $G$ is reserved for an arbitrary reductive algebraic group.  When we refer to a maximal torus we use $T$ and any subtorus is denoted by another capital letter, usually $S$.  Lowercase Greek letters are field elements and other lowercase letters usually denote vectors.  Unless it is the letter $\gamma$, which refers to $(\gamma_1,\gamma_2,\gamma_3) \in (k^*)^3$ or $\diag(\gamma_1,\gamma_2,\gamma_3) \in \GL_3(k)$.  We use $Z(G)$ to denote the center of $G$, $Z_G(S)$ to denote the centralizer of $S$ in $G$.  \

By $\Aut(G)$ we mean the automorphism group of $G$, and by $\Aut(A)$ we mean the linear automorphisms of the Albert algebra, $A$.  The group of inner automorphisms are denoted $\Inn(G)$ and the elements of $\Inn(G)$ are denoted by $\mathcal{I}_g$ where $g\in G$ and $\mathcal{I}_g(x) = gxg^{-1}$. \

We define a \emph{$\theta$-split torus}, $S$, of an involution, $\theta$, as a torus $S  \subset G$ such that $\theta(s)=s^{-1}$ for all $s\in S$.  We call a torus \emph{$(\theta, k)$-split} if it is both $\theta$-split and $k$-split.  If $S$ is a torus that is not necessarily $\theta$-split, we denote by,
\[ S_{\theta}^- = \{ s \in S \ | \ \theta(s)=s^{-1} \}, \]
the elements of $S$ split by $\theta$.  So we can say $S$ is $\theta$-split if and only if $S = S_{\theta}^-$.

Let $S$ be a $\theta$-stable maximal $k$-split torus such that $S_{\theta}^-$ is a maximal $(\theta,k)$-split torus.  In \cite{HW02} it is shown that there exists a maximal $k$-torus $T \supset S$ such that $T_{\theta}^- \supset S_{\theta}^-$ is a maximal $\theta$-split torus.  The involution $\theta$ induces an involution $\tilde{\theta} \in \Aut(X^*(T), X^*(S), \Phi(T), \Phi(S))$.  It was shown by Helminck and Wang \cite{HWa93} that such an involution is unique up to isomorphism.  For $T$ a maximal $k$-torus containing a subtorus $S$,
\[ \tilde{\theta} \in \Aut(X^*(T), X^*(S), \Phi(T), \Phi(S)) \] is \emph{admissible} if there exists an involution $\theta \in \Aut(G,T,S)$ such that $\theta|_{X^*(T)} = \tilde{\theta}$, $S_{\theta}^-$ is a maximal $(\theta, k)$-split torus, and $T_{\theta}^-$ is a maximal $\theta$-split torus of $G$.  This will give us the set of $k$-involutions on $G$ that extend from involutions on the group of characters, $X^*(T)$.  If $\theta$ is a $k$-involution and $S_{\theta}^-$ is a maximal $\theta$-split torus then the elements of the set,
\[ I_k(S_{\theta}^-) = \left\{ s \in S_{\theta}^- \ \big| \ \left(\theta \circ \mathcal{I}_s \right)^2 = \id, \ \left(\theta \circ \mathcal{I}_s\right)(G(k))=G(k)  \right\}, \]
are called \emph{$k$-inner elements} of $\theta$.  Some compositions $\theta \circ \mathcal{I}_s$ will not be isomorphic in the group $\Aut(G)$ for different $s\in I_k(S_{\theta}^-)$, though they will project down to the same involution of the group of characters of a maximal torus fixing the characters associated with a maximal $k$-split subtorus for all $s \in I_k(S_{\theta}^-)$.  

We borrow notation for quadratic forms from the text by Lam, \cite{La05}.  For a $2$-Pfister form we write $\left( \frac{\zeta,\eta}{k} \right)$ for the quadratic form 
\[ q_D(x)= x_0^2 - \zeta x_1^2 -  \eta x_2^2 + \zeta\eta x_3^2, \]
over a field $k$.

\section{Split Albert algebra}

Here we provide notation and background information from the theory of Jordan algebras, especially that of Albert algebras, that we will use in our classification.  We will always think of algebraic groups of type $\F_4$ as the automorphisms of an Albert algebra.  

Our split Albert algebra will be isomorphic to $3 \times 3$ Hermitian matrices over a split octonion algebra defined over a field $k$.  The Albert algebra will have dimension $27$ with respect to the field.  An octonion algebra is \emph{split} if it contain zero divisiors.  We will call the Albert algebra \emph{split} if the octonion algebra associated with the Albert algebra is split.

\subsection{Definition and introduction}

We will let $k$ be a field over characteristic not $2$, for $\ch(k)=2$ see \cite{AS76}, and $C$ a split composition algebra of dimension eight over $k$.  For any fixed $\gamma_i \in k^*$, we will define $A=H(C; \gamma_1,\gamma_2,\gamma_3) = H_3(C,\gamma)$ be the set of $3\times 3$ $\gamma$-Hermitian matrices.  Each $x \in A$, where $f_i \in k$ and $c_j \in C$, will look like
\[
x=h(f_1,f_2,f_3;c_1,c_2,c_3)=
\begin{bmatrix}
f_1 & c_3 & \gamma_1^{-1}\gamma_3\bar{c}_2 \\
\gamma_2^{-1} \gamma_1 \bar{c}_3 & f_2 & c_1 \\
c_2 & \gamma_3^{-1} \gamma_2 \bar{c}_1 & f_3
\end{bmatrix},
\]
where $\bar{x} = q( x,e ) e - x$, in $C$ with $q( \ , \ )$ the bilinear form on $C$.
\\\newline
We will define a product 
\[
xy = \dfrac{1}{2}(x\cdot y + y \cdot x) = \dfrac{1}{2} \left((x+y)^{\cdot 2} - x^{\cdot 2} - y^{\cdot 2} \right) \]
with the dot indicating standard matrix multiplication.

So $A$ is a commutative, nonassociative $k$-algebra of $3 \times 3$ matrices, whose identity element is $e=h(1,1,1;0,0,0)$, the usual $3\times 3$ matrix identity.
\\\newline
We will define a quadratic norm, $Q:A \to k$, with an associated bilinear form 
\[ \langle x,y \rangle = Q(x+y) - Q(x) - Q(y), \]
and we have
\[ Q(x) = \dfrac{1}{2} \left( f_1^2 + f_2^2 + f_3^2 \right) + \gamma_3^{-1}\gamma_2 q(c_1) + \gamma_1^{-1} \gamma_3 q(c_2) + \gamma_2^{-1} \gamma_1 q(c_3). \]
Notice the bilinear form is nondegenerate.  We denote by $H_3(C,\gamma)$, the $3 \times 3$ Hermitian matrices over $C$, where $C$ is an octonion algebra and $\gamma$ is a diagonal matrix with entries in $k^*$, as an \emph{Albert algebra}.

\subsection{Decomposition by a primitive idempotent}

We will follow \cite{SV00} and refer to the family of algebras of $3 \times 3$ $\gamma$-Hermitian matrices over a composition algebra as \emph{$J$-algebras}.   If $w \in A$ and $w^2=w$ then we call $w$ an \emph{idempotent} element.  A $J$-algebra is said to be \emph{reduced} if it contains an idempotent other than zero or the identity.  A $J$-algebra is \emph{proper} if it is isomorphic to an algebra of the form $H_3(C,\gamma)$, where $C$ is a composition algebra.  The idempotent elements of a $J$-algebra play the following role. 
\newtheorem{lem3}[subsubsection]{Lemma}
\begin{lem3}
If $w \in A$ is an idempotent element and $w$ is not $0$ or $e$, then $Q(w) = \frac{1}{2}$ or $1$, $\langle w,e \rangle = 2Q(w)$, $e-u$ is idempotent, $w(e-w)=0$, $\langle w,e-w \rangle = 0$, and $Q(e-w) = \frac{3}{2}- Q(w)$.
\end{lem3}

If $Q(w) = \frac{1}{2}$ then we call $w$ a \emph{primitive idempotent}.  The following theorem can be found in  \cite{Ja68}.

\newtheorem{thm5}[subsubsection]{Theorem}
\begin{thm5}
\label{thm5}
If $A$ is a proper reduced J-algebra, then the composition algebra $C$ where $A \cong H_3(C,\gamma)$ is uniquely determined up to isomorphism.
\end{thm5}

If we fix a primitive idempotent $w \in A$, and let 
	\[ E_0 = \{ a \in e^{\perp} \subset A \ | \ wa =0 \}, \text{ and let } E_1 = \{ a \in e^{\perp} \subset A \ | \ wa = \frac{1}{2} a \}, \]
	then an Albert algebra has a decomposition 
	\[ A = kw \oplus k(e-w) \oplus E_0 \oplus E_1, \] 
called the \emph{Peirce decomposition}.  

\subsection{First Tits construction}

In this paper we are concerned with split Albert algebras, which can be constructed using the following method due to Jacques Tits.  We first need to define a sharped cubic form, $(N, \#, 1_X)$ on a module, $X$, over a ring of scalars with unity.  First we will choose a base point $1_X \in X$, such that $N(1_X)=1$.  Next we define the quadratic map, $\#$, for $x\in X$, such that
\[ x^{\#} = x^2 - \Tr(x)x + \Sr(x)1_X. \]	
We call $\Tr$ the trace form on $X$ and $\Sr$ the quadratic trace.  If the following identities hold for $\#$, $\Tr$, $\Sr$, and $N$,

\begin{align}
	\Tr(x^{\#},y) &= N(x,y) \\
	x^{\#\#} &= N(x)x \\
	1_X \# x &= \Tr(y)1_X - y,
	\end{align}
we call $(N, \#, 1_X)$ a \emph{sharped cubic form}, where 
\[	x \# y = (x+y)^{\#} - x^{\#} - y^{\#}. \]
The maps $\Tr( \ , \ )$, $\Sr( \ , \ )$, and $N( \ , \ )$ are the linearizations of the respective forms.  These will be defined explicitly in our case when needed.  To get an idea of how this works in more generality we refer the reader to \cite{McC05}.  

From a sharped cubic form $(N, \#, 1_X)$ we can construct a unital Jordan algebra, $J(N, \#, 1_X)$, which has unit element $1_X$, and a $U$ operator defined by
	\[ U_xy = \Tr(x,y)x - x^{\#}\# y. \]

\newtheorem{sharpedcubicconstr}[subsubsection]{Proposition, (McCrimmon)}
	\begin{sharpedcubicconstr}
	Any sharped cubic form $(N,\#,1_X)$ gives a unital Jordan algebra $J(N,\#,1_X)$ with unit $1_X$ and product
		\[ xy = \frac{1}{2} ( x \# y + \Tr(x)y + \Tr(y)x - \Sr(x,y)1_X ). \]
	\end{sharpedcubicconstr}

Finally, we relate the sharp product and sharp map to the above forms in the following way,
	\begin{align}
	x\# y &= \{x,y\} - \Tr(x)y - \Tr(y)x + \Sr(x,y)1_X \\
	0 &= x^3 - \Tr(x)x^2 + \Sr(x)x - N(x)1_X.
	\end{align}

From here we define the first Tits construction for Jordan algebras of degree $3$ from associative algebras of degree $3$, a special case of which will provide us with another form of split Albert algebras isomorphic to $H_3(C,\id)$, where $C$ is a split octonion algebra.  

Let $M$ be an associative algebra of degree $3$ over a unital commutative ring $R$ with a cubic norm form $(n, \#, 1_M)$ satisfying the following
	\begin{equation}
	m^3 - \tr(m)m^2 + \sr(m)m - n(m)1_M = 0,
	\end{equation}
with $\tr(m)=n(1_M,m), \sr(m) = n(m,1_M), n(1_M)=1$.  Notice that also if $m^{\#} = m^2 -\tr(m)m + \sr(m)1_M$ we can write
	\begin{equation}
	m \cdot m^{\#} = m^{\#} \cdot m = n(m)1_M,
	\end{equation}

	\begin{equation}
	n(m,m') = \tr(m^{\#},m'),
	\end{equation}

	\begin{equation}
	\tr(m,m') = \tr(mm').
	\end{equation}

\newtheorem{assoccubic}[subsubsection]{Proposition}
\begin{assoccubic}
Let $M$ be a unital associative algebra with a base point such that $n(1_M)=1$, $m^{\#} = m^2 - \tr(m)m + \sr(m)1_M$, and $m \# m' = (m+m')^{\#} - m^{\#} - m'^{\#}$.  Then any cubic form on $M$ satisfies
	\begin{align}
	\tr(1_M) &= \sr(1_M) = 3 \\
	1_M^{\#} &= 1_M \\
	\sr(m,1_M) &= 2\tr(m) \\
	1_M \# m & = \tr(m)1_M - m \\
	\sr(m) &= \tr(m^{\#}) \\
	2\sr(m) &= \tr(m)^2 - \tr(m^2).
	\end{align}
\end{assoccubic}
A proof of this can be found in \cite{McC05}.  For the remainder of the construction we follow Petersson's notes from the Fields Institute workshop on exceptional algebras and groups, but this can also be found in \cite{McC05}.  

\newtheorem{ftc}[subsubsection]{First Tits Construction}
\begin{ftc}
Let $n$ be the cubic norm form of a degree $3$ associative algebra $M$ over $R$, and let $\nu \in R^{*}$.  We define a module 
	\[ J(M,\nu)= M_0 \oplus M_1 \oplus M_2, \]
to be the direct sum of three copies of $M$, and define $1_M$, $N$, $\Tr$, and $\#$ by
	\begin{align}
	1_M &= \id \oplus 0 \oplus 0 \\
	N(m) &= n(m_0)+\nu n(m_1) + \nu^2 n(m_2) - \tr(m_0m_1m_2) \\
	\Tr(m) &= \tr(m_0) \\
	\Tr(m,m') &= \tr(m_0,m'_0) + \tr(m_1,m'_2) + \tr(m_2,m'_1) \\
	m^{\#} &= (m_0^{\#} - \nu m_1m_2) \oplus (\nu m_2^{\#} - m_0m_1) \oplus (m_1^{\#} - m_2m_0) \\
	(mm')^{\#} &= m'^{\#}m^{\#},
	\end{align}
where $n$ and $\tr$ are the norm and trace on $M$ for $m= (m_0 , m_1 , m_2)$ and $m'= (m'_0 , m'_1 , m'_2)$.  Then $(N,\#, 1_M)$ is a sharped cubic form and $J(N,\#,1_M)$ is a Jordan algebra.
\end{ftc}

We want to construct a split Albert algebra over a field $k$.  Let us denote by $\Mat_3(k)$ the $3 \times 3$ matrices over a field $k$.  To perform our construction we pick our associative algebra $\Mat_3(k)$ with $n$ the determinant, $\tr$ is the typical trace of a matrix, $1_M=\id$, and $\nu=1$.  For the remainder of the paper we drop the $\#$ from our notation, and refer the reader to $(3.20)$ for the formula.

\newtheorem{ftcsplit}[subsubsection]{Proposition}
\begin{ftcsplit}
The algebra $J(\Mat_3(k), 1) = A_0 \oplus A_1 \oplus A_2$ is a split Albert algebra.
\end{ftcsplit}

\begin{proof}
The algebra $J(\Mat_3(k),1)$ is reduced, since it contains the primitive idempotent,
\[ \left( \begin{bmatrix}
	1 & \cdot & \cdot \\
	\cdot & \cdot & \cdot \\
	\cdot & \cdot & \cdot
	\end{bmatrix}, 0, 0 \right).
\]
So $J(\Mat_3(k),1) \cong H_3(C,\gamma)$, for some $\gamma$.  Also, $J(\Mat_3(k),1)$ contains a copy of $A_0 \oplus \{0\} \oplus \{0\} \cong \Mat_3(k)^+ \cong H_3(k \oplus k)$, which contains a copy of $k \oplus k$ by \cite{SV00}.  The composition algebra $k\oplus k$ is a split, which implies the octonion algebra $C \supset k \oplus k$, and so $H_3(C,\gamma)$ is split.
\end{proof}

Where $\Mat_3(k)^+$ is the Jordan algebra consisting of the vector space $\Mat_3(k)$ with the typical Jordan product.

\subsection{Automorphisms of an Albert algebra}

For our two presentations of the split Albert algebra over a given field, we will consider two different ways to assure linear maps are automorphisms.  For $H_3(C, \gamma)$ we can just check the map in question is a bijection and respects the Albert algebra multiplication.  When we consider the Albert algebra in the first Tits construction form, we can check that the map is a bijection and respects the norm and adjoint maps, and leaves the base point fixed.

It turns out that, for the most part the automorphisms we want to consider have order $2$, and so we will review some results of Jacobson concerning $\Aut(A)$ found in \cite{Ja68}.

\newtheorem{invjac2}[subsubsection]{Theorem, (Jacobson)}
	\begin{invjac2}
	Let $A$ be a finite dimensional exceptional central simple Jordan algebra.  Then $A$ is reduced if and only if $\Aut(A)$ contains elements of order two.  If the condition holds then any $s \in \Aut(A)$ having order two is either;
	\begin{enumerate}[$(I)$]
	\item a reflection in a sixteen dimensional central simple subalgebra of degree three,
	\item the center element $r_w \neq 1$ in a subgroup $\Aut(A)_w$, $w$ is a primitive idempotent.
	\end{enumerate}
	\end{invjac2}

\newtheorem{corinvjac2}[subsubsection]{Corollary}
	\begin{corinvjac2}
	\label{corinvjac2}
	An involution of type (I)  leaves a subalgebra $B_{K} \subset A_{K} \cong H_3(C,\gamma)$ where $K$ is algebraically closed, and $B_{K} \cong H_3(D, \gamma)$ where $D \subset C$ is a quaternion subalgebra over $K$.  If $A$ is split then any two automorphisms of $A$ of order two and type (II) are conjugate in $\Aut(A)$.  When $k$ is either finite or algebraically closed, then any two automorphisms of $A$ of order two and type (I) are conjugate in $\Aut(A)$.
 \end{corinvjac2}

\section{Automorphisms of groups of type $\F_4$}

In this section we see how an action of $\SL_3 \times \SL_3$ gives us a $k$-split maximal torus in $\Aut(A)$, a construction suggested by H. Petersson following  \cite{Sp298}.  Using the classification from \cite{He00} we use this $k$-split maximal torus to find representatives of the isomorphism classes of $k$-involutions of $\Aut(A)$.

\subsection{Action of $\SL_3(k) \times \SL_3(k)$ on $A \cong J(\Mat_3(k),1)$}

\newtheorem{slact}[subsubsection]{Proposition}
\begin{slact}
The map 
	\[ f_{uv} : A \to A, \text{ where } (a_0, a_1, a_2) \mapsto \left(u a_0 u^{-1}, u a_1 v^{-1}, v a_2 u^{-1} \right), \]
is an automorphism of $A$ if and only if $(u,v) \in \SL_3(k) \times \SL_3(k)$.  
\end{slact}
\begin{proof}
In order for $f_{uv}$ to be an automorphism of $A \cong J(\Mat_3(k), 1)$ we need $f_{uv}$ to preserve the base point, the sharp map and the norm.   A map of the form $f_{uv}$ will clearly fix the basepoint $1_A = ( \id , 0 , 0 )$ for any $u\in \GL_3(k)$.  Next we check that $f_{uv}$ preserves the norm.  So we look at 
	\begin{align*}
		N \left( f_{uv}(a) \right) &= N \left( ( u a_0 u^{-1}, u a_1 v^{-1}, v a_2 u^{-1} ) \right) \\
			&= n\left( u a_0 u^{-1} \right) + n \left( u a_1 v^{-1} \right) + n \left( v a_2 u^{-1} \right) - \tr \left(u a_0 u^{-1} u a_1 v^{-1} v a_2 u^{-1} \right) \\
			&= n(u) n(a_0) n(u^{-1}) + n(u) n(a_1) n(v^{-1}) + n(v) n(a_2) n(u^{-1}) - \tr( u a_0 a_1 a_2 u^{-1} ) \\
			&= n(a_0) + n(a_1) + n(a_2) - \tr ( a_0 a_1 a_2 ) \\
			&= N(a).
	\end{align*}
	
Let us recall $a_i^{-1} = a_i^{\#}n(a_i)^{-1}$ and look at the sharp map
	\[ a^{\#} = \left( a_0^{\#} - a_1a_2 ,  a_2^{\#} - a_0a_1 , a_1^{\#} - a_2a_0 \right). \]
then 
	\begin{align*}
	f_{uv}\left( a^{\#} \right) &= f_{uv} \left( a_0^{\#} - a_1a_2 ,  a_2^{\#} - a_0a_1 , a_1^{\#} - a_2a_0 \right) \\
		&= \left( u(a_0^{\#} - a_1a_2)u^{-1}, u(a_2^{\#} - a_0a_1)v^{-1} , v(a_1^{\#} - a_2a_0)u^{-1} \right) \\
		&= \left( u a_0^{\#} u^{-1} -  u a_1a_2 u^{-1}, u a_2^{\#} v^{-1} -  u a_0a_1 v^{-1} , v a_1^{\#} u^{-1} - v a_2a_0 u^{-1} \right) \\
		&= \left( u a_0^{\#} u^{-1} -  u a_1 v^{-1} v a_2 u^{-1}, u a_2^{\#} v^{-1} -  u a_0 u^{-1} u a_1 v^{-1} , v a_1^{\#} u^{-1} - v a_2 u^{-1} u a_0 u^{-1} \right).
	\end{align*}
Since $a_i^{\#} = a_i^2 - \tr(a_i) a_i + \sr(a_i) c$ and $x^{-1} = x^{\#}$ if and only if $x \in \SL_3(k)$,
	\begin{align*}
	x \left(a_i^{\#} \right) y^{-1}  &= \left( x^{-1} \right) \left( a_i^{\#} \right) \left( y \right)^{\#} \\
		&= \left( x^{-1} \right) \left( y a_i \right)^{\#} \\
		&= \left( y a_i x^{-1} \right)^{\#}.
	\end{align*}
For $a_0$ let $x=y=u$, $a_1$ let $x=u$ and $y=v$, $a_2$ let $x=v$ and $y=u$.
\end{proof}

\subsection{$k$-split maximal torus}

Now if we consider a $k$-split maximal torus in $\SL_3(k)$, 
	\[ T_*= \left\{
	\begin{bmatrix}
	u_1 & & \\
	 & u_1^{-1}u_2 & \\
	  & & u_2^{-1}
	\end{bmatrix}
	\Bigg|
	\ u_1,u_2 \in k \right\} \subset \SL_3(k),
	\]
we see that it is of rank $2$ and so a $k$-split maximal torus $T_* \times T_* \subset \SL_3(k) \times \SL_3(k) \subset \Aut(A,k)$ must have rank $4$ and we have that $\mathcal{T}=T_* \times T_* \subset \Aut(A,k)$ is a $k$-split maximal torus.  In accordance with the action of $f_{uv}$ on $A$, an element of our $k$-split maximal torus can be computed directly.  

\subsection{$k$-involutions}

From the combinatorial classification of invariants of type (1) given in \cite{He00} we know we should have two involutions corresponding to the following diagrams.

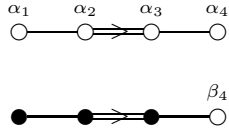
\begin{center} \hspace{25pt}
\begin{picture}(120,24)(10,10)
\put(0,0){\circle{6}}
\put(25,0){\circle{6}}
\put(50,0){\circle{6}}
\put(75,0){\circle{6}}
\put(0,6){\makebox(0,0)[b]{\scriptsize $\alpha_1$}}
\put(25,6){\makebox(0,0)[b]{\scriptsize $\alpha_2$}}
\put(50,6){\makebox(0,0)[b]{\scriptsize $\alpha_3$}}
\put(75,6){\makebox(0,0)[b]{\scriptsize $\alpha_4$}}
\put(3,0){\line(1,0){19}}
\put(28,-1){\line(1,0){19}}
\put(28,1){\line(1,0){19}}
\put(34,-2.5){$>$}
\put(53,0){\line(1,0){19}}
\end{picture}

\vspace{.25cm}

\hspace{25pt}
\begin{picture}(120,24)(10,10)
\put(0,0){\circle*{6}}
\put(25,0){\circle*{6}}
\put(50,0){\circle*{6}}
\put(75,0){\circle{6}}
\put(0,6){\makebox(0,0)[b]{}}
\put(25,6){\makebox(0,0)[b]{}}
\put(50,6){\makebox(0,0)[b]{}}
\put(75,6){\makebox(0,0)[b]{\scriptsize $\beta_4$}}
\put(3,0){\line(1,0){19}}
\put(27.55,-1){\line(1,0){19.55}}
\put(27.55,1){\line(1,0){19.55}}
\put(34,-2.5){$>$}
\put(53,0){\line(1,0){19}}
\end{picture}
\end{center}

\vspace{1cm}

Let $A$ be defined over $k$ the first diagram corresponds to a $k$-involution $\theta:\Aut(A) \to \Aut(A)$ that splits an entire $k$-split maximal torus, and the second diagram corresponds to a $k$-involution $\sigma: \Aut(A) \to \Aut(A)$ that splits a rank $1$ $k$-split torus and fixes a  $k$-split torus of rank $3$.

\subsection{$k$-inner elements}

We have two isomorphism classes of invariants of type (1) as set out by \cite{He00}.  Once we describe the $G(k)$-isomorphism classes of $k$-inner elements for our representatives of admissible $k$-involutions of type (1), our task will be complete.  We will follow \cite{Ja68} and call automorphisms in the equivalence class of $k$-involutions of the form $\theta \circ \I_t$, $t\in I_k(\mathcal{T}_{\theta}^-)$ of type (I), and those of the form $\sigma \circ \I_s$, $s \in I_k(\mathcal{T}_{\sigma}^-)$ of type (II).

By \ref{corinvjac2} all elements of order $2$ in $\Aut(A)$ of type (II) are conjugate for any $k$ where $\ch(k) \neq 2$.  So we have the following theorem.

\newtheorem{sigmaiso}[subsubsection]{Theorem}
\begin{sigmaiso}
\label{sigmaiso}
All $k$-involutions of the form $\sigma \circ \I_s$, $s \in I_k(\mathcal{T}_{\sigma}^-)$ are isomorphic.
\end{sigmaiso}

\begin{proof}
These automorphisms are inner automorphisms of the form $\sigma\circ \I_s(x) = gxg^{-1}$ with $g \in \Aut(A)$ being of order $2$ and leaving an $11$ dimensional subalgebra fixed.  So $g$ is an automorphism of $A$ of type (II), and so the result is immediate from \ref{corinvjac2}.
\end{proof}

The isomorphism classes of $k$-involutions of the form $\theta \circ \I_t$, $t\in I_k(\mathcal{T}_{\theta}^-)$ are said to be of type (I) and leave a $15$ dimensional subalgebra fixed that is isomorphic to $H_3(D,\gamma)$ where $D \subset C$ is a quaternion subalgebra of a split octonion algebra.  In order to classify $k$-involutions of this type we will need the following Lemma.

\newtheorem{order2fixlem}[subsubsection]{Lemma}
\begin{order2fixlem}
\label{order2fixlem}
Let $\mathcal{A}$ be a $k$-algebra and $\mathcal{D}$ and $\mathcal{D}'$ subalgebras of $\mathcal{A}$.  If $t,t' \in \Aut(\mathcal{A})$ are elements of order $2$ and $t,t'$ fix $\mathcal{D},\mathcal{D}'$ elementwise respectively, then $t \cong t'$ if and only if $\mathcal{D} \cong \mathcal{D}'$ over $k$.
\end{order2fixlem}

\begin{proof}
Let $\mathcal{D},\mathcal{D}' \subset \mathcal{A}$ such that $t(a)=a$ and $t'(a') = a'$ for all $a \in \mathcal{D}$ and $a' \in \mathcal{D}'$, and let $\mathcal{D}$ and $\mathcal{D}'$ be the largest such subalgebras with respect to $t$ and $t'$.  First we will show sufficiency.  Let $\mathcal{D} \cong \mathcal{D}'$, and let $g\in \Aut(\mathcal{A})$ be such that $g(\mathcal{D}')=\mathcal{D}$.  For $c \in \mathcal{A}$ $c = a+b$ where $a \in \mathcal{D}$ and $b\in \mathcal{A} - \mathcal{D}$.  Since $t,t'$ are of order $2$ we have $\mathcal{D}$ (resp. $\mathcal{D}$) is the $1$-space and $\mathcal{A} - \mathcal{D}$ (resp. $\mathcal{A}-\mathcal{D}'$) is the $(-1)$-space of $t$ (resp. $t'$).  Then we have,
\begin{align*}
gt'g^{-1}(a+b) &= gt'(a'+b') \\
	&= g(a'-b') \\
	&=a - b \\
	&= t(a+b),
\end{align*}
since $g(\mathcal{A}-\mathcal{D}') = \mathcal{A}-\mathcal{D}$.   To show necessity we start by assuming there exists a $g\in \Aut(\mathcal{A})$ such that $gt'g^{-1} = t$, which implies that $t'g^{-1} = g^{-1}t$, and from this we see that
\begin{align*}
t'g^{-1}(a+b) &= g^{-1}t(a+b) \\
	t'(g^{-1}(a) + g^{-1}(b) ) &= g^{-1}(a-b) \\
	t'(g^{-1}(a) + g^{-1}(b)) &= g^{-1}(a) - g^{-1}(b).
\end{align*}
This shows us that $g^{-1}(a) \in \mathcal{D}'$ for all $a\in \mathcal{D}$, and so $\mathcal{D} \cong \mathcal{D}'$.
\end{proof}

We say that two diagonal matrices, $\gamma, \gamma' \in (k^*)^3$, are \emph{equivalent} and write $\gamma \sim_C \gamma'$ if $H_3(C,\gamma) \cong H_3(C,\gamma')$, where $C$ is a composition algebra.

\newtheorem{subalgA}[subsubsection]{Proposition}
\begin{subalgA}
Let $D \subset C$ be a quaternion subalgebra of the octonion algebra $C$, over $k$.  If $H_3(D,\delta),H_3(D',\delta') \subset H_3(C,\gamma)$ then $H_3(D, \delta) \cong H_3(D', \delta')$ if and only if $D$ and $D'$ are split, or $D$ and $D'$ are division algebras and $\delta \sim_C \delta' \in k^*/q_D(C)^*$.
\end{subalgA}

\begin{proof}
From \cite{SV00} every automorphism of a subalgebra of $H(C,\gamma)$ extends to an automorphism of $H(C,\gamma)$.  This fact along with \ref{thm5} gives us the result.
\end{proof}

If we consider $\theta \circ \I_{t}$ of type (I), notice that $\theta = \I_{g}$ where,
\[ g =
\begin{bmatrix}
p & \cdot & \cdot \\
\cdot & \cdot & p \\
\cdot & p & \cdot
\end{bmatrix}, 
\]
where $\I_p :\Mat_9(k) \to \Mat_9(k)$, and $p(x) = x^T$.  So the $k$-inner elements are of the form $t(u_1,u_2,v_1,v_2) = t \in \mathcal{T}_{\theta}^- = \{ t \in \mathcal{T} \ | \ \theta(t)=t^{-1} \}$ and $\theta \circ \I_t$ is a $k$-involution.  In this case $\mathcal{T}_{\theta}^-$ is a maximal torus, and all elements $t \in \mathcal{T}$ are such that $\theta \circ \I_t$ is a $k$-involution for all $k$.  To check that this is true we can simply notice that $\theta \circ \I_t = \I_g \circ \I_t = \I_{gt}$, and $(gt)^2 = \id$ for all $t \in \mathcal{T} = \mathcal{T}_{\theta}^-$.

If we let $t=t(u_1,u_2,v_1,v_2)$ then we can compute $A^{\theta}$, the subalgebra of $A$ fixed by the element of $\Aut(A)$ that induces the $k$-involution, $\theta \circ \I_t$.

We assume that our field is not of characteristic $2$ and we use the quadratic form $\frac{1}{2}\Tr(x^2)$ from \cite{Sp60}, to make an identification between $H_3(D,\gamma) \subset H_3(C,\id)$, and $A^{\theta} \subset J(\Mat_3(k),1)$.  We need to recall a theorem of Jacobson \cite{Ja68}.

\newtheorem{jacisoext}[subsubsection]{Theorem, (Jacobson)}
\begin{jacisoext}
Let $H\subset J$ and $H'\subset J$ be reduced simple subalgebras of degree $3$ of a reduced simple exceptional Jordan algebra $J$, if there is an isomorphism $H\cong H'$ then it can be extended to an automorphism in $J$.
\end{jacisoext}

\newtheorem{jacisoextcor1}[subsubsection]{Corollary}
\begin{jacisoextcor1}
Let $C$ be an octonion algebra and $D\subset C$ a quaternion subalgebra, then $H_3(D,\delta)$ is a subalgebra of $H_3(C,\gamma)$ if and only if $H_3(C,\gamma) \cong H_3(C,\delta)$.
\end{jacisoextcor1}

\newtheorem{jacisoextcor2}[subsubsection]{Corollary}
\begin{jacisoextcor2}
Let $C$ be a split octonion algebra with quaternion subalgebra $D \subset C$ then $H_3(D,\delta)$ is a subalgebra for all $\delta \in (k^*)^3$.
\end{jacisoextcor2}
\begin{proof}
This follows from the fact that there is only one isomorphism class of algebras of the form $H_3(C,\gamma)$ when $C$ is split.
\end{proof}

With this in mind we look at isomorphism classes of algebras of the form $H_3(D,\gamma) \subset H_3(C,\id)$ where $C$ is a split octonion algebra.

\newtheorem{mainlem}[subsubsection]{Lemma}
\begin{mainlem}
\label{mainlem}
Let $C$ be a split octonion algebra with quadratic form $q$ and $D$ a quaternion subalgebra with quadratic form $q_D$.  Then
\begin{enumerate}[$(1)$]
\item if $k=K$ is algebraically closed there is one isomorphism class of the form $H_3(D,\gamma)$,
\item if $k=\mathbb{F}_p$ where $\ch(p) >2$ there is one isomorphism class of algebras of the form $H_3(D,\gamma)$,
\item if $k=\mathbb{R}$ there are $3$ isomorphism classes of algebras of the form $H_3(D,\gamma)$ corresponding to $D$ being split, and $D$ being a division algebra with $\gamma=\id$ or $\gamma=(-1,1,1)$,
\item if $k=\mathbb{Q}_p$ there are $2$ isomorphism classes of algebras of the form $H_3(D,\gamma)$ corresponding to $D$ being split or $D$ being a division algebra.
\item if $k=\mathbb{Q}$ there are an infinite number of isomorphism classes.
\end{enumerate}
\end{mainlem}
\begin{proof}
For (1) and (2) there are only split quaternion algebras, and therefore only split algebras of the form $H_3(D,\gamma)$.  For (3) when $k=\mathbb{R}$ it is well known, see for example \cite{SV00} chapter 1, there are 2 isomorphism classes of quaternion algebras.  If $D$ is split there is one isomorphism class of algebras of the form $H_3(D,\gamma)$.  If $D$ is a division algebra the isomorphism classes are determined by norm classes $k^*/q_D(D)^*$, 5.8.1 \cite{SV00}.  If $k=\mathbb{R}$ then $k^*/q_D(D)^* = \{\pm1\}$, which give us two equivalency classes of diagonal matrices of the form $\diag(\gamma_1,\gamma_2,\gamma_3)$, where $\gamma_i \in \mathbb{R}^*$.  One can be represented by $\diag(1,1,1)$ and the other by $\diag(-1,1,1)$.  For (4) again we get two isomorphism classes of quaternion algebras, and since $q_D(D)$ represents all values in $\mathbb{Q}_p$, see \cite{Se97, SV00}, so $k^*/q_D(D)^*=\{1\}$ and all matrices of the form $\diag(\gamma_1,\gamma_2,\gamma_3) \in (k^*)^3$ are equivalent.  To see (5) notice that there are an infinite number of isomorphism classes of quaternion division algebras over $\mathbb{Q}$ and in order for two $J$-algebras to be isomorphic their composition algebras must be isomorphic.
\end{proof}

\newtheorem{mainthm}[subsubsection]{Theorem}
\begin{mainthm}
Let $t=t(u_1,u_2,v_1,v_2)$ and let $gt \in \Aut(A)$ be an involution of type (I), where $A$ is a split Albert algebra.  Then for $k$-involutions of the form $\theta\circ\I_t = \I_{gt}$,
\begin{enumerate}[$(1)$]
\item if $k=K$ is algebraically closed there is one isomorphism class,
\item if $k=\mathbb{F}_p$ where $\ch(p) >2$ there is one isomorphism class,
\item if $k=\mathbb{R}$ there are $3$ isomorphism classes,
\item if $k=\mathbb{Q}_p$ there are $2$ isomorphism classes,
\item if $k=\mathbb{Q}$ there are an infinite number of isomorphism classes.
\end{enumerate}
\end{mainthm}
\begin{proof}
This follows from \ref{order2fixlem} and \ref{mainlem}.
\end{proof}

In fact we can identify a representative of isomorphism classes of $k$-involutions of the type $\theta \circ \I_t$ for each field.  For $k$ algebraically closed or a finite field we can take $u_1=u_2=v_1=v_2=1$.  To find representatives for the isomorphism classes of $k$-involutions when $k$ is $\mathbb{R}, \mathbb{Q}_p$, or $\mathbb{Q}$ it helps to consider the quadratic form from \cite{SV00}.  If we consider an Albert algebra $A$ over $C$, a split composition algebra, then $A \cong H_3(C,\id)$.  Then an element of $H_3(C,\id)$ is of the form,
\[ x =
\begin{bmatrix}
f_1 & x_3 & \bar{x}_2 \\
\bar{x}_3 & f_2 & x_1 \\
x_2 & \bar{x}_1 & f_3
\end{bmatrix},
\]
where $f_i \in k$ and $x_l \in C$, and $\bar{ \ }$ is the algebra involution in $C$.  The quadratic form described in \cite{SV00} is of the form
\[ Q(x) =  \frac{1}{2}(f_1^2 + f_2^2 + f_3^2) + q(x_1) + q(x_2) + q(x_3), \]
with $q$ the quadratic form on $C$.  We consider the same quadratic form on the subalgebra $H_3(D,\gamma) \subset H_3(C,\id)$.  An element of $H_3(D,\gamma)$ has the form
\[ y=
\begin{bmatrix}
f_1 & y_3 & \gamma_1^{-1}\gamma_3\bar{y}_2 \\
\gamma_2^{-1}\gamma_1\bar{y}_3 & f_2 & y_1 \\
y_2 & \gamma_3^{-1}\gamma_2\bar{y}_1 & f_3,
\end{bmatrix}
\]
where $f_i \in k$, $y_l \in D$ a quaternion subalgebra of $C$, and $(\gamma_1,\gamma_2, \gamma_3) \in (k^*)^3$.  The quadratic form restricted to this subalgebra looks like,
\[ Q(y) = \frac{1}{2}(f_1^2 + f_2^2 +f_3^2) + \gamma_3^{-1}\gamma_2 q_D(y_1) + \gamma^{-1}\gamma_3 q_D(y_2) + \gamma_2^{-1}\gamma_1 q_D(y_3), \]
with $q_D$ the quadratic form on $C$ restricted to a quaternion subalgebra $D \subset C$.  From here we need two facts about equivalent quadratic forms.  We know that $A^{\theta} \cong H_3(D,\gamma)$ for some $D$ and $\gamma$.  So, if we compute $Q(a)$, for $a \in A^{\theta}$,
\begin{align*}
Q(a)= = &\frac{1}{2}(a_{01}^2 + a_{05}^2 + a_{09}^2 ) \\
&+ u_1^{-2}u_2a_{02}^2 + u_2 v_1 a_{17}^2 + u_2v_1^{-1}v_2a_{18}^2 + u_2v_2^{-1} a_{19}^2 \\
&+ u_1^{-1}u_2^{-1}a_{03}^2 + u_1u_2^{-1} v_1 a_{14}^2 + u_1u_2^{-1} v_1^{-1} v_2 a_{15}^2 + u_1u_2^{-1}v_2^{-1} a_{16}^2 \\
&+ u_1u_2^{-2}a_{06}^2 + u_1^{-1} v_1 a_{11}^2 + u_1^{-1}v_1^{-1}v_2a_{12}^2 + u_1^{-1}v_2^{-1} a_{13}^2
\end{align*}

\newtheorem{compquad}[subsubsection]{Proposition, \cite{Ja58}}
\begin{compquad}
A quaternion algebra is completely determined by its quadratic form.
\end{compquad}

We can use this along with the fact that the quadratic form of a quaternion algebra is completely determined by a $2$-Pfister form,
\[ \left( \frac{\zeta, \eta}{k} \right), \]
where $\zeta$ and $\eta$ are the negative squares of two basis vectors in $e^{\perp} \subset D$, where $e$ is the identity element in $D$.  It is known that
\begin{equation}
\label{qquadeq}
\left( \frac{\zeta, \eta}{k} \right) \cong  \left( \frac{m^2\zeta, n^2\eta}{k} \right), 
\end{equation}
see \cite{GS06} 1.1.2, where $m,n \in k^*$.  It is also helpful to note here that 
\begin{equation}
\label{gammaeq}
(\delta \gamma_1, \delta \gamma_2, \delta\gamma_3) \sim (\gamma_1,\gamma_2\gamma_3) \sim (\delta_1^2 \gamma_1, \delta_2^2 \gamma_2, \delta_2^2\gamma_3),
\end{equation}
where $\delta, \delta_i \in k^*$, \cite{Ja68}.  Using \ref{qquadeq} and \ref{gammaeq} we can rewrite,
\begin{align*}
Q(a)=  &\frac{1}{2}(a_{01}^2 + a_{05}^2 + a_{09}^2 ) \\
&+ u_2 \left(a_{02}^2 + v_1 a_{17}^2 + v_1^{-1}v_2a_{18}^2 + v_2 a_{19}^2\right) \\
&+ u_1^{-1}u_2^{-1}\left(a_{03}^2 + v_1 a_{14}^2 +  v_1^{-1} v_2 a_{15}^2 + v_2 a_{16}^2\right) \\
&+ u_1\left(a_{06}^2 + v_1 a_{11}^2 + v_1^{-1}v_2a_{12}^2 + v_2 a_{13}^2 \right),
\end{align*}
making the identifications $u_2 \mapsto \gamma_3^{-1}\gamma_2$, $u_1^{-1}u_2^{-1} \mapsto \gamma_1^{-1}\gamma_3$, $u_1 \mapsto \gamma_2^{-1}\gamma_1$, $\zeta \mapsto v_1$, and $v_1^{-1}v_2 \mapsto \eta$ we have equivalent quadratic forms.

\newtheorem{u1u2v1v2}[subsubsection]{Proposition}
\begin{u1u2v1v2}
For the following fields, $k$, we can take as representatives of isomorphism classes of $k$-involutions of $\Aut(A)$ to be of the form $\I_{gt}$ where $g$ is defined above, and $t=t(u_1,u_2,v_1,v_2)$
\begin{enumerate}[$(1)$]
\item $k=K$ or $k=\mathbb{F}_p$ where $p>2$, $t=t(1,1,1,1)$ is a representative of the only isomorphism class,
\item $k=\mathbb{R}$ for $D$ split we can choose $t(1,1,-1,1)$, for the positive definite case we can choose $t(1,1,1,1)$, and for the indeterminate quadratic form we can choose $t(-1,1,1,1)$,
\item $k=\mathbb{Q}_2$ we can choose $t(1,1,-1,1)$ for the split case and $t(1,1,1,1)$ for $D$ a division algebra,
\item $k=\mathbb{Q}_p$ with $p>2$ we can choose $t(1,1,-1,1)$ for $D$ split, and $t(1,1,-p,-Z_p)$ for $D$ a division algebra.
\end{enumerate}
\end{u1u2v1v2}
\begin{proof}
(1) is straight forward.  (2) is well known, but this can be seen through straight forward computations and the fact that $(1,1,1) \not\sim (-1,1,1)$. Recall that there are only two isomorphism classes of quaternion algebras.  (3) follows from straight forward computations and the fact that there are only two isomorphism classes of quaternion algebras over $k=\mathbb{Q}_2$ is determined by $\left(\frac{-1,-1}{\mathbb{Q}_2}\right)$ a division algebra, and a quaternion algebra determined by $\left(\frac{1,-1}{\mathbb{Q}_2}\right)$ for the split case.  (4)  can be seen if we let $\mathbb{Q}_p^*/(\mathbb{Q}_p^*)^2 = \{1, p, Z_p, pZ_p\}$, where $Z_p$ is the smallest non-square in $\mathbb{F}_p$.  There are two isomorphism classes of quaternion algebras over $\mathbb{Q}_p$ when $p>2$; one when the quaternion algebra is determined by $\left(\frac{1,-1}{\mathbb{Q}_p}\right)$ for the split case, and the other is determined by $\left(\frac{p,Z_p}{\mathbb{Q}_p}\right)$ for the division algebra case.
\end{proof}

For $k=\mathbb{Q}$ we have seen that there are an infinite number of isomorphism classes of $k$-involutions of $\Aut(C)$ where $C$ is an octonion algebra, \cite{Hu12}.  This is due to the fact that there are an infinite number of isomorphism classes of quaternion division algebras of $\mathbb{Q}$.  We recall that 
\[ \left(\frac{-1,p}{\mathbb{Q}} \right) \not\cong \left(\frac{-1,q}{\mathbb{Q}} \right), \]
when $p$ and $q$ are distinct primes both equivalent to $3 \mod 4$.  This alone is enough to give us subalgebras of $H(C,\id)$ of the form $H_3(D_i,\id)$ where $D_i$ is the quaternion algebra with the quadratic form $\left(\frac{-1,p_i}{\mathbb{Q}} \right)$ where $p_i$ are all distinct primes equivalent to $3 \mod 4$.

\subsection{A decomposition}

The $k$-involutions of the form $\theta \circ \I_t = \I_s$ as defined above, correspond to an element of order $2$ in $\Aut(A)$ fixing a subalgebra of the form $H_3(D,\gamma)$, where $D$ is a quaternion subalgebra defined over $k$, and $\gamma$ is a diagonal matrix with entries in $k^*$. 

This induces a decomposition of the Albert algebra $A$ into a Jordan algebra over the quaternion algebra fixed by an element of order $2$ in $\Aut(C)$ and a quaternion multiple of the skew symmetric matrices taken over a quaternion algebra of the same type.

Let $t \in \Aut(A)$ be of order $2$ such that $t|_C = \hat{t} \in \Aut(C)$ is of order $2$, then $s$ is of the form,
\[ t(f_1,f_2f_3,c_1,c_2,c_3) = (f_1,f_2,f_3,\hat{t}(c_1),\hat{t}(c_2),\hat{t}(c_3)), \]
and fixes a quaternion subalgebra $D \subset C$ that is either split or a division algebra.  Then the map $s$ fixes a subalgebra of the form $H_3(D,\gamma) \subset H_3(C,\id) \cong A$.  So the algebra $H_3(C,\id)$ decomposes as follows
\[ H_3(D,\gamma) \oplus \Skew_3(D, \gamma)\cdot j, \]
where we think of $j=\diag(j,j,j)$ with $j \in D^{\perp}$, so that $C=D\oplus Dj$.  An element of $H_3(C,\id)$ can be written in the form
\[ X + Y \cdot j, \]
where $X \in H_3(D,\gamma)$ and $Y \in \Skew_3(D,\gamma)$.  Then we can look at the elements $t \in \Aut(A)$ that fix a subalgebra isomorphic to $H_3(D,\gamma)$
\[ t(X + Y\cdot j) = t(X) + t(Y\cdot j), \]
since $t$ leaves $H_3(D,\gamma)$ invariant and is an automorphism, it must leave $\Skew_3(D,\gamma) = H_3(D,\gamma)^{\perp}$ invariant as well.  This means we can think of $t \in \Aut(A)$ in terms of its action on the subalgebra and its perpendicular complement.  We will rename the map
\[ t|_{H_3(D,\gamma)} := r, \]
and the map $t(Y\cdot j) = s(Y) \cdot j$, which allows us to write,
\[ t(X + Y\cdot j) = r(X) + s(Y) \cdot j. \]

\newtheorem{decompmult1}[subsubsection]{Propsition}
\begin{decompmult1}
For $X \in H_3(D,\gamma)$ and $Y,V \in \Skew_3(D,\gamma)$ the following are true,
\begin{enumerate}[\hspace{1cm}$1.$]
\item $X(Y \cdot j) \in \Skew_3(D,\gamma)$,
\item $(Y\cdot j)(V\cdot j)  \in H_3(D,\gamma)$.
\end{enumerate}
\end{decompmult1}
\begin{proof}
This can be shown easily through straight forward computation.
\end{proof}

We denote the products defined above by
\[  (X \bullet Y)\cdot j := X(Y \cdot j), \text{ and } Y*V:= (Y\cdot j)(V\cdot j). \]

Using the above notation we can then say that
\begin{align*}
t\left((X + Y\cdot j)(U+V\cdot j) \right) &= t \left( XU + Y*V + (X\bullet V + U \bullet Y)\cdot j \right)  \\
&= r(XU + Y*V) + s (X\bullet V + U \bullet Y)\cdot j \\
&= r(XU) + r(Y*V) + \left( s(X \bullet V) + s(U \bullet Y) \right) \cdot j 
\end{align*}
If we now use the fact that $t \in \Aut(A)$, we can say
\begin{align*}
t(X + Y\cdot j)s(U+V\cdot j) &= \left( r(X) + s(Y)\cdot j \right) \left( r(U) + s(V) \cdot j \right) \\
&= r(X)r(U) + s(Y)*s(V) + \left( r(X) \bullet s(V) + r(U) * s(Y) \right) \cdot j .
\end{align*}
Setting $t\left((X + Y\cdot j)(U+V\cdot j) \right) = t(X + Y\cdot j)t(U+V\cdot j)$, we see the following
\begin{align}
 r(XU) &= r(X)r(U)  \label{autDeq1} \\
r(Y*V) &= s(Y)*s(V) \label{autDeq2} \\
s(X\bullet V) &= r(X) \bullet t(V). \label{autDeq3} 
\end{align}

 Define an \emph{algebra involution} to be a map on a $k$-algebra $\mathcal{A}$ such that
\begin{enumerate}[\hspace{1cm}1.]
\item $\iota(x + y) = \iota(x) + \iota(y)$,
\item $\iota(xy) = \iota(y)\iota(x)$,
\item $\iota^2(x) = x$,
\end{enumerate}
for all $x,y \in \mathcal{A}$.  We will be concerned with the case where the $k$-algebra $\Mat_3(D)$, and $\iota = \iota_{\gamma} :\Mat_3(D) \to \Mat_3(D)$ is induced by $\gamma = \diag(\gamma_1,\gamma_2,\gamma_3)$ with $\gamma_i \in k^*$, by $\iota(x) = \gamma^{-1} \bar{x}^T \gamma$.   

If we look at a typical element of the space $\Skew_3(D,\gamma) * \Skew_3(D,\gamma) \subset H_3(D,\gamma)$ we need to consider the product $(Y\cdot j)(V \cdot j) = Y * V$, where $Y, V \in \Skew_3(D,\gamma)$, $j \in D^{\perp}$.  We have an element of the following form 
\begin{equation}
\label{skewskew}
 Y*V = \frac{q(j)}{2} \left( \gamma^{-1} \left( \iota(V) \cdot Y + \iota(Y) \cdot V  \right)^T \gamma \right). 
\end{equation}

Now we consider the space $H_3(D,\gamma)) \bullet \Skew_3(D,\gamma) \subset \Skew_3(D,\gamma)$, and look at elements of the form $X\bullet V$ where $X \in H_3(D,\gamma)$ and $V \in \Skew_3(D,\gamma)$.  For this we will consider elements $X \in H_3(D,\gamma)$, i.e.
\[ X =
\begin{bmatrix}
f_1 & x_3 & \gamma_1^{-1}\gamma_3 \bar{x}_2 \\
\gamma_2^{-1}\gamma_1 \bar{x}_3 & f_2 & x_1 \\
x_2 & \gamma_3^{-1}\gamma_2 \bar{x}_1 & f_3 
\end{bmatrix}.
\]
We can look at the product $X \bullet V$, where $X(V\cdot j) = (X \bullet V) \cdot j$, and we arrive at
$X\bullet V =  \left(V \cdot \overline{X} + (V^T \cdot X^T )^T \right)$.

\subsection{Fixed point groups}

Over any field there is only one isomorphism class of $k$-involutions, which has as a representative conjugation by an element in $\Aut(A)$ fixing an $11$  dimensional subalgebra, so there is only one isomorphism class of their fixed point groups by \cite{He00}.

We will take $\sigma$ as defined above as our representative of this class of $k$-involutions.  We will call the $11$ dimensional subalgebra $B \subset A$, and so the subgroup of $G=\Aut(A)$ that leaves $B$ invariant is $G^{\sigma}$.

\newtheorem{spin}[subsubsection]{Proposition}
\begin{spin}
The fixed point group of $\sigma$ is isomorphic to $\Spin(Q, E_0)$.
\end{spin}

\begin{proof}
The subgroup $\Aut(A)_w \subset \Aut(A)$ that leaves a primitive idempotent $w$ invariant is isomorphic to $\Spin(Q,E_0)$ where $Q$ is the quadratic trace form on $A$ and $E_0$ is the $0$-space of multiplication by $w$.  The algebra $E_0$ is isomorphic to the $11$ dimensional algebra fixed by $\sigma$.
\end{proof}

\newtheorem{autm}[subsubsection]{Proposition}
\begin{autm}
The fixed point group of $\theta \circ \I_t$ is isomorphic to \\ $\Aut(\Mat_3(D),\iota) \times \Sp(1)$ where $\iota$ is an algebra involution on $\Mat_3(D)$. 
\end{autm}

\begin{proof}
Let $r \in \Aut(H_3(D,\gamma))$ for $D \subset C$ a quaternion subalgebra of $C$ and $\gamma \in (k^*)^3$.   By \cite{Ma67}, $r$ extends uniquely to an element $\tilde{r} \in \Aut(\Mat_3(D),\iota)$, where $\iota = \iota_{\gamma}$ is the algebra involution in $A$ induced by $\gamma$.  If $t \in \Aut(H_3(C,\id))$ such that $t^2 = \id$ and $t$ leaves elementwise fixed a subalgebra of the form $H(D,\gamma)$.  If $X \in H_3(D,\gamma) \subset \Mat_3(D)$ and $Y \in \Skew_3(D,\gamma) \subset \Mat_3(D)$, then $t(X+Y\cdot j) = r(X) + s(Y) \cdot j$ where $j\in D^{\perp}$ with $q(j) \neq 0$, $r \in \Aut(H_3(D,\gamma))$ and $s \in L(\Skew_3(D,\gamma))$.  Let $x,y \in D$ then $t|_C (x+yj) = r|_D(x) + s|_D(y)j$ such that $r|_D \in \Aut(D)$ and $s|_D = p (r|_D)$, where $p \in \Sp(1)$, \cite{Hu12}.  So $s|_D \in \Aut(D) \times \Sp(1)$, and we have $s = \tilde{p}r \in \Aut(\Mat_3(D), \iota) \times H$ and $\det(\tilde{p})=1$.  But we have $\Aut(\Mat_3(D),\iota) \times H \subset \Aut(H_3(C,\id))$, which has rank $4$. The subgroup $\Aut(H_3(D,\gamma)) \cong \Aut(\Mat_3(D),\iota)$ is of type $\C_3$, $H \supset \Sp(1)$ has rank $1$, and so $H \cong \Sp(1)$.
\end{proof}

These groups correspond to a description of coordinates given by Kac in \cite{Ka81}, and what Serre calls Kac coordinates in \cite{Se06}.  We now provide a summary of the idea of Kac coordinates from \cite{Se06} starting with the theory for $k$ having characteristic zero.  These are also mentioned in \cite{Le12}.  We fix a maximal torus $T$, and a set of roots $\Phi(T)$ with base $(\alpha_i)_{i \in I}$, with $\alpha_i \in X^*(T) \otimes_{\mathbb{Z}} \mathbb{Q}$.  We denote by
\[ \tilde{\alpha} = \displaystyle\sum_{i\in I} \lambda_i\alpha_i, \]
the longest root.  The coefficients $\lambda_i \in \mathbb{Z}$ and $\lambda_i\geq 1$.  If we then take $I_0 = I \cup \{0\}$, and set $\alpha_0 = - \tilde{\alpha}$ we have
\[ \displaystyle\sum_{i \in I_0} \lambda_i \alpha_i = 0. \]
We can associate the set $I$ to the set of vertices of the Dynkin diagram of $\Phi(T)$, and $I_0$ to the set of vertices of the extended Dynkin diagram.  When $\ch(k)=0$ we choose a parametrization of the roots of unity.  In general we want a homomorphism $\epsilon:\mathbb{Q} \to K^*$, where $\ker(\epsilon) = \mathbb{Z}$.  For example the natural choice when $K=\mathbb{C}$ is the map $\epsilon(\chi) = \e^{2\pi i \chi}$.

From this we can associate an element $t_{\chi} \in T(k)$ in the following way
\[ \omega(t_\chi) = \epsilon\left( \displaystyle\sum_{i \in I} \xi_i(\omega) \chi_i \right), \]
where $\omega \in X^*(T)$, and $\xi_i(\omega)$ are the coordinates of $\omega$ with respect to $(\alpha_i)$.  When $Z(G)$ contains only the identity this is enough to characterize $t_\chi$, and this is the case for $\Aut(A)$.  We define the set $P$ in the following way,
\[ P = \left\{ \chi=(\chi_i) \ | \ \chi_i \geq 0, \sum \lambda_i \chi_i = 1, i \in I_0 \right\}. \]

\newtheorem{KacCoor}[subsubsection]{Theorem, Kac}
\begin{KacCoor}
Any element of finite order  of $G(k)$ is conjugate to exactly one $t_\chi$, with $\chi \in P$.
\end{KacCoor}

This is proven in \cite{Ka81, Le12}.  We can simplify the situation when looking for conjugacy classes of element of a fixed order $\kappa$.  Let $\chi_i = \frac{\rho_i}{\kappa}$, and the equation defining $P$ becomes 
\begin{equation}
\label{KacEq}
\displaystyle\sum_{i \in I_0} \lambda_i \rho_i = \kappa. 
\end{equation}

In our case $\Aut(A)$ if of type $\F_4$, and choosing a popular set of roots and base we have that
\[ \tilde{\alpha} = 2\alpha_1 + 3 \alpha_2 + 4 \alpha_3 + 2 \alpha_4. \]
So the elements of order $2$ correspond to the solution $(\rho_i)_{i \in I_0}$ with $\rho_i \in \mathbb{N}$ such that
\[ \rho_0 + 2\rho_1 + 3 \rho_2 + 4 \rho_3 + 2 \rho_4 = 2, \]
and we have the solutions $(0,1,0,0,0)$ and $(0,0,0,0,1)$.  

Also pointed out in \cite{Se06} is that the centralizers of each $t_\chi$ have Dynkin diagrams contained within the affine Dynkin diagram for $G$ whose vertices correspond to the $\rho_i=0$.  In our case we will delete the $\alpha_1$ vertex for one conjugacy class, and the $\alpha_4$ vertex for another conjugacy class.

\begin{center}
\begin{picture}(120,24)(10,10)
\put(0,0){\circle{6}}
\put(25,0){\circle{6}}
\put(50,0){\circle{6}}
\put(75,0){\circle{6}}
\put(100,0){\circle{6}}
\put(0,6){\makebox(0,0)[b]{\scriptsize $\alpha_0$}}
\put(25,6){\makebox(0,0)[b]{\scriptsize $\alpha_1$}}
\put(50,6){\makebox(0,0)[b]{\scriptsize $\alpha_2$}}
\put(75,6){\makebox(0,0)[b]{\scriptsize $\alpha_3$}}
\put(100,6){\makebox(0,0)[b]{\scriptsize $\alpha_4$}}
\put(3,0){\line(1,0){19}}
\put(28,0){\line(1,0){19}}
\put(53,-1){\line(1,0){19}}
\put(53,1){\line(1,0){19}}
\put(59,-2.5){$>$}
\put(78,0){\line(1,0){19}}
\end{picture}
\end{center}

\vspace{1.5cm}

This leaves us with the following two Dynkin diagrams.

\begin{center}
\begin{picture}(120,24)(10,10)
\put(0,0){\circle{6}}
\put(50,0){\circle{6}}
\put(75,0){\circle{6}}
\put(100,0){\circle{6}}
\put(0,6){\makebox(0,0)[b]{\scriptsize $\alpha_0$}}
\put(50,6){\makebox(0,0)[b]{\scriptsize $\alpha_2$}}
\put(75,6){\makebox(0,0)[b]{\scriptsize $\alpha_3$}}
\put(100,6){\makebox(0,0)[b]{\scriptsize $\alpha_4$}}
\put(52.6,-1){\line(1,0){19.4}}
\put(52.6,1){\line(1,0){19.4}}
\put(59,-2.5){$>$}
\put(78,0){\line(1,0){19}}
\end{picture}

\vspace{1cm}

\begin{picture}(120,24)(10,10)
\put(0,0){\circle{6}}
\put(25,0){\circle{6}}
\put(50,0){\circle{6}}
\put(75,0){\circle{6}}
\put(0,6){\makebox(0,0)[b]{\scriptsize $\alpha_0$}}
\put(25,6){\makebox(0,0)[b]{\scriptsize $\alpha_1$}}
\put(50,6){\makebox(0,0)[b]{\scriptsize $\alpha_2$}}
\put(75,6){\makebox(0,0)[b]{\scriptsize $\alpha_3$}}
\put(3,0){\line(1,0){19}}
\put(28,0){\line(1,0){19}}
\put(53,-1){\line(1,0){19}}
\put(53,1){\line(1,0){19}}
\put(59,-2.5){$>$}
\end{picture}
\end{center}

\vspace{1cm}

The first is the Dynkin diagram of a group of type $\A_1 \times \C_3$ corresponding to the fixed point group of type (I), $\Sp(1) \times \Aut(M,\iota)$, and the second to a group of type $\B_4$ corresponding to the fixed point group of type (II), $\Spin(Q,E_0)$.

\subsection{Galois cohomology}

The interpretation of conjugacy classes of $k$-involutions in terms of Galois cohomology follows much the same way as \cite{Hu12}.  If we first consider $H^0(Gal_k, \Aut(A,K))$ to be the cohomology group of $\Aut(A)$ defined over $K$ with coefficients in $Gal_k$ the absolute Galois group of $k$.  In this case $H^0(Gal_k, \Aut(A,K)) \cong \Aut(A,k)$ the automorphism group of $A$ defined over $k$. The group $H^1(Gal_k,\Aut(A,K))$  is the group of $K/k$-forms of $A$.  So when we consider the $k$-involutions $\Aut(A)$ we are looking at the subgroup of automorphisms that fix a certain subalgebra of $A$.  These subalgebras come in two types.  If $t \in \Aut(A)$ is an element of order $2$, then $t$ induces a $k$-involution $\I_t$, and $t$ fixes a subalgebra of the form $H_3(D,\gamma)$ or $E_0$, the zero space of multiplication by some idempotent element in $A$.  In either case, if we let $B \subset A$, the group $H^1(Gal_k,\Aut(A,B,K))$ corresponds to the $K/k$-forms of $B$.

\newtheorem{gal}[subsubsection]{Proposition}
\begin{gal}
\label{gal}
Let $A$ be a $k$-algebra, and $B$ a subalgebra of $A$ fixed by an element of order $2$ in $\Aut(A)$, then there is a bijection between $\mathcal{C}_k$, the isomorphism classes of involutions of $\Aut(A)$, and $H^1(Gal_k, \Aut(A,B,K))$ the $K/k$-forms of $B$.
\end{gal}
\begin{proof}
This follows from \ref{order2fixlem}.
\end{proof}

When $B \cong H_3(D,\gamma)$ the cohomology group $H^1(Gal_k,\Aut(A,B,K))$ corresponds to the $K/k$-forms, which is in bijection with isomorphism classes of $k$-involutions of $\Aut(A)$ by \ref{gal}.  This also corresponds to the $K/k$-forms of fixed point groups, which we can think of as the isomorphism classes of the centralizer of the $k$-involution in $\Aut(A)$.  In other words the groups $H^1(Gal_k,\Aut(A,B,K)$ and $H^1(Gal_k, Z_G(\varphi))$ are in bijection when $\varphi$ is an element of order $2$ in $\Aut(G)$, where $G \cong \Aut(A)$.

\bibliographystyle{plain}

\end{document}